\documentclass{amsart}
\usepackage{amsmath}
\usepackage[psamsfonts]{amssymb}
\usepackage[dvips]{graphics}
\usepackage[all]{xy}
\newtheorem{Theorem}{Theorem}[section]
\newtheorem{Lemma}[Theorem]{Lemma}

\theoremstyle{definition}

\theoremstyle{remark}
\newtheorem{Remark}[Theorem]{Remark}

\def\labelenumi{\rm ({\makebox[0.8em][c]{\theenumi}})}
\def\theenumi{\arabic{enumi}}
\def\labelenumii{\rm {\makebox[0.8em][c]{(\theenumii)}}}
\def\theenumii{\roman{enumii}}

\renewcommand{\thefigure}
{\arabic{section}.\arabic{figure}}
\makeatletter
\@addtoreset{figure}{section}
\def\@thmcountersep{-}
\makeatother

\numberwithin{equation}{section}

\begin{document} 

\title[]{On calculations of the twisted Alexander ideals for spatial graphs, handlebody-knots and surface-links}

\author{Atsushi Ishii}
\address{Institute of Mathematics, University of Tsukuba, 1-1-1 Tennodai, Tsukuba, Ibaraki 305-8571, Japan }
\email{aishii@math.tsukuba.ac.jp}
\thanks{The first author was supported by JSPS KAKENHI Grant Number 24740037.}

\author{Ryo Nikkuni}
\address{Department of Mathematics, School of Arts and Sciences, Tokyo Woman's Christian University, 2-6-1 Zempukuji, Suginami-ku, Tokyo 167-8585, Japan}
\email{nick@lab.twcu.ac.jp}
\thanks{The second author was supported by JSPS KAKENHI Grant Number 24540094.}

\author{Kanako Oshiro}
\address{Deparment of Information and Communication Sciences, Sophia University, 7-1 Kioicho, Chiyoda-ku, Tokyo 102-8554, Japan}
\email{oshirok@sophia.ac.jp}
\thanks{The third author was supported by JSPS KAKENHI Grant Number 25800052.}

\subjclass{Primary 57M05; Secondary 57M15, 57Q45}

\date{}


\keywords{Spatial graph, Handlebody-knot, Surface-link, twisted Alexander ideal}

\begin{abstract}
There are many studies about twisted Alexander invariants for knots and links, but calculations of twisted Alexander invariants for spatial graphs, handlebody-knots, and surface-links have not been demonstrated well. In this paper, we give some remarks to calculate the twisted Alexander ideals for spatial graphs, handlebody-knots and surface-links, and observe their behaviors. For spatial graphs, we calculate the invariants of Suzuki's theta-curves and show that the invariants are nontrivial for Suzuki's theta-curves whose Alexander ideals are trivial. For handlebody-knots, we give a remark on abelianizations and calculate the invariant of the handlebody-knots up to six crossings. For surface-links, we correct Yoshikawa's table and calculate the invariants of the surface-links in the table.

\end{abstract}

\maketitle

\section{Introduction}\label{sect:introduction}

The Alexander ideal is a knot invariant derived from the fundamental group of the exterior of a knot with an abelianization, which can be specified by the meridian of the knot. The twisted Alexander ideal is a generalization of the Alexander ideal, which is derived from the fundamental group, an abelianization and a group representation. There are two versions of the twisted Alexander invariant introduced by Lin \cite{Lin01} and Wada \cite{Wada94}. In this paper, we follow Wada's version. The Alexander ideal can be defined not only for a knot but also for a finitely presentable group with an abelianization. The twisted Alexander ideal can also be defined for a finitely presentable group with an abelianization and a group representation. Since calculations of twisted Alexander ideals have not been demonstrated well except for knots, we give some remarks to calculate the twisted Alexander ideals for spatial graphs, handlebody-knots and surface-links, and observe their behaviors. 

For spatial graphs, we focus on Suzuki's theta-$n$ curve $\Theta_n$ as illustrated in Fig. \ref{Suzuki_theta_n} where $n$ is a positive integer satisfying $n\ge 3$. Alexander ideals for $\Theta_n$ were calculated in \cite{Sato10,Suzuki84} and as we will demonstrate later, it is trivial if $n\equiv1,5\pmod{6}$. In this case, we define a family of group representations and give a formula of the twisted Alexander ideals for $\Theta_n$. It follows from the formula that Suzuki's theta-$n$ curve is nontrivial if $n\equiv1,5\pmod{6}$, although it can not be shown by the Alexander ideal.

For handlebody-knots, we focus on the handlebody-knots in the table \cite{IshiiKishimotoMoriuchiSuzuki12} of genus two handlebody-knots up to six crossings. Since a meridian system of a handlebody-knot is not unique, a group representation and an abelianization can not be specified via meridian systems. Then we sum up the twisted Alexander ideals over representations and abelianizations, and obtain an invariant of matrix form. We confirm that the twisted Alexander ideal works better than the number of representations as expected.

For surface-links, we focus on the surface-links in Yoshikawa's table \cite{Yoshikawa94}, where each surface-link is represented by a ch-diagram. We correct three calculations in his table and calculate the twisted Alexander ideals for surface-links in the table. We also remark that there exists a pair of two surface-links such that their twisted Alexander ideals differ but that their Alexander ideals coincide. 

This paper is organized as follows. In Section 2, we review the twisted Alexander ideal for a finitely presentable group $G$ associated with an epimorphism from $G$ to an abelian group and a group representation from $G$ to a matrix group. In Section 3, we give a formula of the twisted Alexander ideal for Suzuki's theta-curves. In Section 4, we introduce an invariant of matrix form and give a table of the invariants for genus two handlebody-knots up to six crossings. In Section 5, we calculate the twisted Alexander ideal for surface-links in Yoshikawa's table and correct three calculations in his table. In this paper, we denote by $\mathbb{Z}_p$ the cyclic group of order $p$, namely $\mathbb{Z}_p=\mathbb{Z}/p\mathbb{Z}$.

\section{Twisted Alexander ideals}\label{sect:}

In this section, we give a brief review of the twisted Alexander ideals for finitely presentable groups. Let $P$ be a commutative ring with unity $1$. For two matrices $A,B$ over $P$, we say that $A$ and $B$ are {\it elementarily equivalent} if they are transformed into each other by a finite sequence of the {\it elementary operations}: 
\begin{enumerate}
\item Permuting rows and columns, 
\item Adjoining a row of zeros; $M\to\begin{pmatrix} M \\ \boldsymbol{0} \end{pmatrix}$, 
\item Adding to some row a linear combination of other rows, 
\item Adding to some column a linear combination of other columns, 
\item Adding a new row and a new column such that the entry in the intersection of the new row and new column is $1$, and the all of the remaining entries in the new column are zero; $M\to\begin{pmatrix} M & \boldsymbol{0} \\ * & 1 \end{pmatrix}$, 
\item The inverses of (1), (2), (3), (4) and (5). 
\end{enumerate}
Then we denote by $A\sim B$. Note that the multiplication of a row and column by a unit of $P$ can be realized by elementary operations. 
Let $A$ be a $\infty\times s$ matrix over $P$ such that only the first $t$ rows contain nonzero entries. Then the \textit{$d$-th elementary ideal} $E_d\left(A\right)$ of $A$ is defined by the ideal of $P$ generated by all $(s-d)$-minors of $A$ if $0< s-d \le t$, $(0)$ if $s-d>t$ and $(1) = P$ if $s-d\le 0$. It is known that elementarily equivalent two matrices over $P$ have the same sequence of elementary ideals \cite{CrowellFox77}. 

For a group $G$ and a ring $R$, we denote the group ring of $G$ over $R$ by $RG$. Let $F_{s}$ be the free group with rank $s$ generated by $x_{1},x_{2},\ldots,x_{s}$. The {\it Fox's free derivative} \cite{Fox53} with respect to $x_i$ is an additive map 
\begin{eqnarray*}
\frac{\partial}{\partial x_{i}}:\mathbb{Z}F_{s}\to\mathbb{Z}F_{s}
\end{eqnarray*} 
which satisfies 
\begin{eqnarray*}
\frac{\partial x_j}{\partial x_i} = 
\begin{cases}
e & (i=j) \\
0 & (i\neq j)
\end{cases},\ \ 
\frac{\partial (uv)}{\partial x_i}
&=&\frac{\partial u}{\partial x_i}+u\frac{\partial v}{\partial x_i}
\end{eqnarray*}
for $u,v\in F_{s}$. In particular, the map satisfies 
\begin{eqnarray*}
\frac{\partial x_{i}^{p}}{\partial x_i}
= 
\begin{cases}
e+x_{i}+x_{i}^{2}+\cdots+x_{i}^{p-1} & (p>0) \\
0 & (p=0) \\
-x_{i}^{-1}-x_{i}^{-2}-\cdots-x_{i}^{p} & (p<0). 
\end{cases}
\end{eqnarray*}

For $r_{1},r_{2},\ldots,r_{t}\in F_{s}$, let 
\begin{eqnarray*}
G = \langle x_{1},x_{2},\ldots,x_{s}\ |\ r_{1},r_{2},\ldots,r_{t}\rangle
\end{eqnarray*}
be a finitely presented group with $s$ generators and $t$ relators. Let $\phi:F_{s}\to G$ be the canonical epimorphism and $\alpha$ an epimorphism from $G$ to an abelian group 
\begin{eqnarray*}
G_0=\big\langle t_1,\ldots,t_r\ |\ t_1^{k_1},\ldots,t_r^{k_r},\ \left[t_{i},t_{j}\right]\ (1\le i<j\le r)\big\rangle
\end{eqnarray*}
 for some non-negative integers $k_1,\ldots,k_r$, where $\left[t_{i},t_{j}\right]$ denotes the commutator of $t_{i}$ and $t_{j}$. Then for a ring $R$, the group ring $RG_{0}$ may be identified with the quotient ring of Laurent polynomial ring ${R\left[t_1^{\pm1},\ldots,t_r^{\pm1}\right]}/{\big(t_1^{k_1}-1,\ldots,t_r^{k_r}-1\big)}$. We denote by $\tilde{\phi}$ and $\tilde{\alpha}$ the linear extensions of $\phi$ and $\alpha$ from $\mathbb{Z}F_{s}$ to $\mathbb{Z}G$ and from $\mathbb{Z}G$ to $\mathbb{Z}G_0$, respectively. 
Then we call the $\infty\times s$ matrix 
\begin{eqnarray*}
A\left(G,\tilde{\alpha}\right) = 
\left(
       \begin{array}{c}
       \tilde{\alpha}\circ\tilde{\phi}\left({\displaystyle \frac{\partial r_i}{\partial x_j}}\right) \\
       O 
       \end{array}
       \right)
\end{eqnarray*}
each of whose entries belongs to ${{\mathbb Z}\left[t_1^{\pm1},\ldots,t_r^{\pm1}\right]}/{\big(t_1^{k_1}-1,\ldots,t_r^{k_r}-1\big)}$ the {\it Alexander matrix of $G$ associated with $\alpha$}, where $O$ stands for $\infty\times s$ zero matrix. We call the $d$-th elementary ideal $E_d\left(A\left(G,\tilde{\alpha}\right)\right)$ of $A\left(G,\tilde{\alpha}\right)$ the {\it $d$-th Alexander ideal of $G$ associated with $\alpha$}. 

Moreover, let $GL\left(n;R\right)$ be the general linear group of degree $n$ over a ring $R$ and $\rho:G\to GL(n;R)$ a group representation. We denote by $\tilde{\rho}$ the linear extension of $\rho$ from $\mathbb{Z}G$ to the matrix ring $M_n\left(R\right) = M\left(n,n;R\right)$. Then the tensor product homomorphism 
\begin{eqnarray*}
\tilde{\rho}\otimes\tilde{\alpha}:\mathbb{Z}G\to 
M_n\big({R\left[t_1^{\pm1},\ldots,t_r^{\pm1}\right]}/{\big(t_1^{k_1}-1,\ldots,t_r^{k_r}-1\big)}\big)
\end{eqnarray*}
is defined by 
\begin{eqnarray*}
\left(\tilde{\rho}\otimes\tilde{\alpha}\right)\left(\sum r_{i}g_{i}\right)
= \sum r_{i}{\alpha}\left(g_{i}\right){\rho}\left(g_{i}\right)\ (r_{i}\in {\mathbb Z},\ g_{i}\in G). 
\end{eqnarray*}
Then we call the $\infty\times s$ matrix 
\begin{eqnarray*}
A\left(G,\tilde{\rho}\otimes\tilde{\alpha}\right)
= \left(
       \begin{array}{c}
       (\tilde{\rho}\otimes\tilde{\alpha})\circ\tilde{\phi}\left({\displaystyle \frac{\partial r_i}{\partial x_j}}\right) \\
       O 
       \end{array}
       \right)
\end{eqnarray*}
each of whose entries belongs to $M_{n}\big({R\left[t_1^{\pm1},\ldots,t_r^{\pm1}\right]}/{\big(t_1^{k_1}-1,\ldots,t_r^{k_r}-1\big)}\big)$ the {\it twisted Alexander matrix of $G$ associated with $\alpha$ and $\rho$}, where $O$ stands for $\infty\times s$ matrix consisting entirely of $n\times n$ zero matrix. We regard $A\left(G,\tilde{\rho}\otimes\tilde{\alpha}\right)$ as $\infty\times ns$ matrix in $M\big(nt,ns;{R\left[t_1^{\pm1},\ldots,t_r^{\pm1}\right]}/{\big(t_1^{k_1}-1,\ldots,t_r^{k_r}-1\big)}\big)$ where only the first $nt$ rows contain nonzero entries. Then we call the $d$-th elementary ideal $E_d\left(A\left(G,\tilde{\rho}\otimes\tilde{\alpha}\right)\right)$ of $A\left(G,\tilde{\rho}\otimes\tilde{\alpha}\right)$ the \textit{$d$-th twisted Alexander ideal of $G$ associated with $\rho$ and $\alpha$}. 

Since any two presentations of $G$ are related by Tietze transformations and these transformations induce a sequence of elementary operations, it follows that the Alexander matrix $A\left(G,\tilde{\alpha}\right)$ of $G$ associated with $\alpha$ does not depend on the choice of a presentation of $G$. Namely we have the following. 

\begin{Theorem}\label{alex_id}
The sequence of Alexander ideals of $G$ associated with $\alpha$ does not depend on the choice of a presentation of $G$. \hfill $\square$ 
\end{Theorem}

In a similar way, the $d$-th twisted Alexander matrix of $G$ associated with $\rho$ and $\alpha$ does not depend on the choice of a presentation of $G$. In addition, let $\rho':G\to GL\left(n;R\right)$ be a group representation which is conjugate to $\rho$, that is, there exists $B\in GL\left(n;R\right)$ such that $\rho'(x)=B\rho(x)B^{-1}$ for any $x\in G$. Then it is not hard to see that the twisted Alexander matrix $A\left(G,\tilde{\rho}'\otimes\tilde{\alpha}\right)$ of $G$ associated with $\alpha$ and $\rho'$ is elementarily equivalent to the twisted Alexander matrix $A\left(G,\tilde{\rho}\otimes\tilde{\alpha}\right)$ of $G$ associated with $\alpha$ and $\rho$. Therefore we have the following. 

\begin{Theorem}\label{tw_alex_id}
The sequence of twisted Alexander ideals of $G$ associated with $\alpha$ and $\rho$ does not depend on the choice of a presentation of $G$ and a representative element in the conjugacy class of $\rho$. \hfill $\square$ 
\end{Theorem}

In particular for $t=s-1$ and $G_{0}=\langle t_{1}\ |\ \emptyset\rangle$, it is known that the first Alexander ideal $E_{1}\left(A\left(G,\tilde{\alpha}\right)\right)$ is always principal and its generator is called the {\it Alexander polynomial} of $G$ associated with $\alpha$ \cite{CrowellFox77}. Moreover, a specific generator of $E_{n}\left(A\left(G,\tilde{\rho}\otimes\tilde{\alpha}\right)\right)$ produces the {\it twisted Alexander polynomial} of $G$ associated with $\alpha$ and $\rho$. We refer the reader to \cite{Wada94,Lin01,FriedlVidussi11} for the precise definition of the twisted Alexander polynomial.

\section{Spatial graphs}

Let $\Gamma$ be a finite and labeled graph embedded in the $3$-sphere $S^{3}$. Then $\Gamma$ is called a \textit{spatial graph}. Two spatial graphs are said to be \textit{equivalent} if there exists an orientation-preserving self-homeomorphism on $S^{3}$ which sends one to the other as labeled graphs. A spatial graph $\Gamma$ is said to be \textit{trivial} if there exists a spatial graph $\Gamma'$ contained in a $2$-sphere in $S^{3}$ such that $\Gamma$ is equivalent to $\Gamma'$. For a spatial graph $\Gamma$, the fundamental group of the spatial graph complement $G_{\Gamma}=\pi_{1}\left(S^{3}-\Gamma\right)$ is finitely presentable and we can obtain a group presentation of deficiency $1-\beta_{0}\left(\Gamma\right)+\beta_{1}\left(\Gamma\right)$, where $\beta_{i}\left(\Gamma\right)$ denotes the $i$-th Betti number of $\Gamma$ \cite{Kinoshita72}. In particular, if $\Gamma$ is trivial, then $G_{\Gamma}$ is isomorphic to the free group of rank $\beta_{1}\left(\Gamma\right)$. Let $l$ be a $1$-dimensional cycle (in the sense of homology) with integral coefficients on $\Gamma$. Then we define a homomorphism $\alpha_{l}$ from $G_{\Gamma}$ to $\left\langle t\,|\,\emptyset\right\rangle$ by $\alpha_{l}\left(g\right)=t^{\mathrm{lk}\left(g,l\right)}$ for any element $g$ in $G_{\Gamma}$, where $\mathrm{lk}$ denotes the linking number in $S^{3}$.
Then the collection of the Alexander ideals of $G_{\Gamma}$ associated with $\alpha_{l}$ is an invariant of $\Gamma$. Moreover, let $\rho$ be a group representation from $G_{\Gamma}$ to $SL\left(n;R\right)$. Then the collection of the twisted Alexander ideals of $G_{\Gamma}$ associated with $\alpha_{l}$ and all possible $\rho$ is also an invariant of $\Gamma$. In particular the following holds by the definition of the elementary ideals. 

\begin{Lemma}\label{trivia}
If $\Gamma$ is trivial, then it follows that
\begin{eqnarray}\label{trivial}
E_{d}\left(A\left(G_{\Gamma},\tilde{\alpha}_{l}\right)\right)
=\begin{cases}
\left(0\right) & (d<\beta_{1}\left(\Gamma\right)) \\
\left(1\right) & (d\ge\beta_{1}\left(\Gamma\right))
\end{cases},
\end{eqnarray}
\begin{eqnarray}\label{trivial2}
E_{d}\left(A\left(G_{\Gamma},\tilde{\rho}\otimes\tilde{\alpha}_{l}\right)\right)
=\begin{cases}
\left(0\right) & (d<n\beta_{1}\left(\Gamma\right)) \\
\left(1\right) & (d\ge n\beta_{1}\left(\Gamma\right))
\end{cases}
\end{eqnarray}
for any $l$ and $\rho$. \hfill $\square$
\end{Lemma}

For a positive integer $n\ge 3$, let $\Theta_{n}$ be {\it Suzuki's theta-$n$ curve} \cite{Suzuki84} which is a spatial graph represented by  the diagram illustrated in Fig. \ref{Suzuki_theta_n}. Note that $\Theta_{3}$ is also called {\it Kinoshita's theta curve}. We denote $G_{\Theta_{n}}$ by $G_{n}$ simply.

\begin{figure}[htbp]
      \begin{center}
\scalebox{0.5}{\includegraphics*{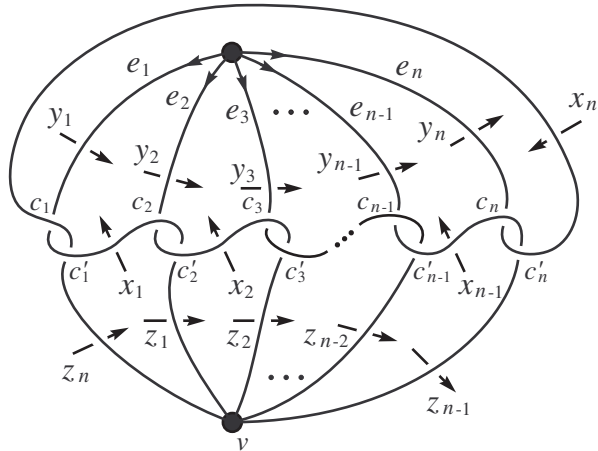}}
      \end{center}
   \caption{Suzuki's theta-$n$ curve}
  \label{Suzuki_theta_n}
\end{figure} 

It can be easily seen that any proper spatial subgraph of $\Theta_{n}$ is trivial. The non-triviality of $\Theta_{n}$ for all $n$ was shown in \cite{scharlemann92} first by a geometric argument and also was shown in \cite{Livingston95} by an application of a symmetry of $\Theta_{n}$ and the branched cover of $S^{3}$. Our purpose in this section is to demonstrate an effectiveness of (twisted) Alexander ideals by showing the non-triviality of $\Theta_{n}$ as an application of it. 

\begin{Lemma}\label{Su}{\rm (\cite{Suzuki84})} 
$G_{n}$ has a presentation 
\begin{eqnarray*}
\left\langle x_{1},x_{2},\ldots,x_{n}\,|\, \left(x_{1}x_{n}x_{1}^{-1}\right)\left(x_{2}x_{1}x_{2}^{-1}\right)
\cdots\left(x_{n}x_{n-1}x_{n}^{-1}\right)\right\rangle.
\end{eqnarray*}
\end{Lemma}

\begin{proof}
By taking a Wirtinger presentation on the diagram in Fig. \ref{Suzuki_theta_n}, we have 
\begin{eqnarray*}
G_{n}\cong \left\langle x_{i},y_{i},z_{i}\ (i=1,2,\ldots,n)\,|\,r_{i},r'_{i},r''\ (i=1,2,\ldots,n)\right\rangle,
\end{eqnarray*}
where $r_{i}$, $r'_{i}$ and $r''$ are relators corresponding to the crossings $c_{i}$, $c'_{i}$ and vertex $v$ as illustrated in Fig. \ref{Suzuki_theta_n}, respectively. Moreover we have 
\begin{eqnarray*}
r_{i} &=& x_{i}x_{i-1}y_{i}^{-1}x_{i-1}^{-1}\ (i=1,2,\ldots,n), \\
r'_{i} &=& z_{i-1}x_{i}x_{i-1}^{-1}x_{i}^{-1}\ (i=1,2,\ldots,n), \\
r'' &=& z_{n}z_{1}z_{2}\cdots z_{n-1},
\end{eqnarray*}
where suffix $i$ is taken modulo $n$. Then by deleting generators $y_{i},z_{i}$ and relators $r_{i},r'_{i}$ by Tietze transformations, we have the result.
\end{proof}

We denote the $1$-dimensional cycle $e_{i}-e_{n}$ on $\Theta_{n}$ by $l_{i}$ for $i=1,2,\ldots,n-1$. Note that $\left\{l_{1},l_{2},\ldots,l_{n-1}\right\}$ is a basis of $H_{1}\left(\Gamma;{\mathbb Z}\right)$, namely $\beta_{1}\left(\Gamma\right)=n-1$. Then $l=\sum_{i=1}^{n-1}l_{i}$ is also a $1$-dimensional cycle on $\Theta_{n}$. We denote the homomorphism $\alpha_{l}$ by $\alpha_{n}$. For generators $x_{1},x_{2},\ldots,x_{n}$ of $G_{n}$, we have
\begin{eqnarray*}
\alpha_{n}\left(x_{i}\right)&=& t^{{\rm lk}\left(x_{i},l\right)}=t^{{\rm lk}\left(x_{i},l_{i}\right)}=t\ (i=1,2,\ldots,n-1),\\
\alpha_{n}\left(x_{n}\right)&=& t^{{\rm lk}\left(x_{n},l\right)}=t^{\sum_{i=1}^{n-1}{\rm lk}\left(x_{n},l_{i}\right)}=t^{1-n}. 
\end{eqnarray*}
Then Sato calculated the Alexander ideal of $A\left(G_{n},\tilde{\alpha}_{n}\right)$ as follows.

\begin{Theorem}\label{suzuki_ideal2}{\rm (Sato \cite{Sato10})}
\begin{eqnarray*}
E_{d}\left(A\left(G_{n}, \tilde{\alpha}_{n}\right)\right)=\left\{
   \begin{array}{@{\,}ll}
       \left(0\right) & \mbox{$(d<n-1)$} \\
       \left(1-t+t^{2}, 1-t^{n}\right) & \mbox{$(d=n-1)$} \\
       \left(1\right) & \mbox{$(d \ge n)$}
   \end{array}
   \right. .
\end{eqnarray*}
In particular for $d=n-1$, it follows that 
\begin{eqnarray*}
E_{n-1}\left(A\left(G_{n},\tilde{\alpha}_{n}\right)\right)=\left\{
   \begin{array}{@{\,}ll}
       \left(1\right) & \mbox{$\left(n \equiv 1, 5 \pmod{6}\right)$} \\
       \left(3,1+t\right) & \mbox{$\left(n \equiv 2, 4 \pmod{6}\right)$} \\
       \left(2, 1-t+t^{2}\right) & \mbox{$\left(n \equiv 3 \pmod{6}\right)$} \\
       \left(1-t+t^{2}\right) & \mbox{$\left(n \equiv 0 \pmod{6}\right)$}
   \end{array}
   \right. . 
\end{eqnarray*}
\end{Theorem}

In \cite{Sato10}, Sato calculated Alexander ideals for a wide class of spatial graphs containing $\Theta_{n}$. In the following, we describe a proof of Theorem \ref{suzuki_ideal2} for reader's convenience. 

\begin{proof}
By Lemma \ref{Su}, it is clear that $E_{d}\left(A\left(G_{n},\tilde{\alpha}_{n}\right)\right)=(0)$ if $d< n-1$ and $(1)$ if $d\ge n$. So we calculate $E_{n-1}\left(A\left(G_{n},\tilde{\alpha}_{n}\right)\right)$. We put 
\begin{eqnarray*}
r=\left(x_{1}x_{n}x_{1}^{-1}\right)\left(x_{2}x_{1}x_{2}^{-1}\right)
\cdots\left(x_{n}x_{n-1}x_{n}^{-1}\right).
\end{eqnarray*}
Then by a direct calculation we have 
\begin{eqnarray}\label{fd1}
\frac{\partial r}{\partial x_{i}}
&=&\left\{\prod_{m=1}^{i-1}\left(x_{m}x_{m-1}x_{m}^{-1}\right)\right\}
\left\{e-\left(x_{i}x_{i-1}x_{i}^{-1}\right)
+\left(x_{i}x_{i-1}x_{i}^{-1}\right)x_{i+1}\right\}
\end{eqnarray}
for $i=1,2,\ldots,n-1$ and
\begin{eqnarray}\label{fd2}
\frac{\partial r}{\partial x_{n}}
&=& x_{1}+\prod_{m=1}^{n-1}\left(x_{m}x_{m-1}x_{m}^{-1}\right)-r.
\end{eqnarray}
Thus we have
\begin{eqnarray}
\tilde{\alpha}_{n}\circ \tilde{\phi}\left(\frac{\partial r}{\partial x_{i}}\right)
&=& t^{i-1-n}\left(1-t+t^{2}\right)\ (i\neq 1,n-1),\label{fd3}\\
\tilde{\alpha}_{n}\circ \tilde{\phi}\left(\frac{\partial r}{\partial x_{1}}\right)
&=& t^{1-n}\left(t^{n-1}-1+t\right),\label{fd4}\\
\tilde{\alpha}_{n}\circ \tilde{\phi}\left(\frac{\partial r}{\partial x_{n-1}}\right)
&=& t^{-n}\left(t^{n-2}-t^{n-1}+1\right). \label{fd5}
\end{eqnarray}
Note that 
\begin{eqnarray}
1-t^{n} &=& \left(1-t+t^{2}\right)-t\left(t^{n-1}-1+t\right), \label{r1} \\
t^{n-2}-t^{n-1}+1 &=& t^{-1}\left(1-t^{n}\right)
+t^{-1}\left(t^{n-1}-1+t\right). \label{r2}
\end{eqnarray}
Then by (\ref{fd3}),(\ref{fd4}),(\ref{fd5}),(\ref{r1}) and (\ref{r2}), it follows that
\begin{eqnarray*}
E_{n-1}\left(A\left(G_{n}, \tilde{\alpha}_{n}\right)\right)
&=& \left(t^{n-1}-1+t,\ 1-t+t^{2},\ t^{n-2}-t^{n-1}+1\right)\\
&=& \left(1-t^{n},\ t^{n-1}-1+t,\ 1-t+t^{2},\ t^{n-2}-t^{n-1}+1\right)\\
&=& \left(1-t^{n},\ t^{n-1}-1+t,\ 1-t+t^{2}\right)\\
&=& \left(1-t^{n},\ 1-t+t^{2}\right). 
\end{eqnarray*}
Namely we obtain the first half of the theorem. Next we show the second half. Note that the reminder of dividing $1-t^{n}$ by $1-t+t^{2}$ equals $0$ if $n\equiv 0 \pmod{6}$, $1-t$ if $n\equiv 1 \pmod{6}$, $2-t$ if $n\equiv 2 \pmod{6}$, $2$ if $n\equiv 3 \pmod{6}$, $1+t$ if $n\equiv 4 \pmod{6}$ and $t$ if $n\equiv 5 \pmod{6}$. Therefore we have 
\begin{eqnarray*}
\left(1-t+t^{2},1-t^{n}\right)=
\left\{
   \begin{array}{@{\,}ll}
       \left(1-t+t^{2}\right) & \mbox{$\left(n\equiv 0 \pmod{6}\right)$} \\
       \left(1-t+t^{2},1-t\right)=(1) & \mbox{$\left(n\equiv 1 \pmod{6}\right)$} \\
       \left(1-t+t^{2},2-t\right)=(3,1+t) & \mbox{$\left(n\equiv 2 \pmod{6}\right)$} \\
       \left(1-t+t^{2},2\right)  & \mbox{$\left(n\equiv 3 \pmod{6}\right)$} \\
       \left(1-t+t^{2},1+t\right)=(3,1+t) & \mbox{$\left(n\equiv 4 \pmod{6}\right)$} \\
       \left(1-t+t^{2},t\right)=(1) & \mbox{$\left(n\equiv 5 \pmod{6}\right)$}
   \end{array}
   \right. . 
\end{eqnarray*}
This completes the proof. 
\end{proof}

\begin{Remark}
Let $\alpha'_{n}$ be the homomorphism from $G_{n}$ to $\left\langle t\,|\,t^{n}\right\rangle$ defined by $\alpha'_{n}\left(x_{i}\right)=t\ (i=1,2,\ldots,n)$. Note that $\alpha'_{n}=\xi_{n}\circ\alpha_{n}$, where $\xi_{n}$ is the canonical homomorphism from $\left\langle t\,|\,\emptyset\right\rangle$ to $\left\langle t\,|\,t^{n}\right\rangle$. Then Theorem \ref{suzuki_ideal2} leads to the following result of Suzuki \cite{Suzuki84}: 
\begin{eqnarray*}
E_{d}\left(A\left(G_{n}, \tilde{\alpha}'_{n}\right)\right)=\left\{
   \begin{array}{@{\,}ll}
       \left(0\right) & \mbox{$(d<n-1)$} \\
       \left(1-t+t^{2}\right) & \mbox{$(d=n-1)$} \\
       \left(1\right) & \mbox{$(d \ge n)$}
   \end{array}
   \right. . 
\end{eqnarray*}
Note that if $n \equiv 1, 5 \pmod{6}$ then $1-t+t^{2}$ is invertible in ${\mathbb Z}\left[t,t^{-1}\right]$ and therefore $E_{d}\left(A\left(G_{n}, \tilde{\alpha}'_{n}\right)\right)=(1)$ (this was pointed out in \cite{McAteeSilverWilliams01} first).
\end{Remark}

By Lemma \ref{trivia} and Theorem \ref{suzuki_ideal2}, it follows that $\Theta_{n}$ is nontrivial for $n\equiv 0,2,3,4\pmod{6}$. In the case of $n\equiv 1,5\pmod{6}$, Theorem \ref{suzuki_ideal2} does not work for showing the nontriviality of $\Theta_{n}$. In the following, we apply twisted Alexander ideals to $\Theta_{n}$ in the case of $n\equiv 1,5\pmod{6}$.

\begin{Lemma}\label{rho_homo}
Let $n$ be a positive integer satisfying $n\ge 3$ and $n\equiv 1,5\pmod{6}$.
Let $\rho$ be a map from $G_{n}$ to $SL\left(2;{\mathbb Z}_{2}\right)$ defined by
\begin{eqnarray*}
\rho\left(x_{i}\right) &=& \begin{pmatrix} 0 & 1 \\ 1 & 1 \end{pmatrix}\ (i=1,2,\ldots,6k+3),\\
\rho\left(x_{6k+4}\right) &=& \begin{pmatrix} 0 & 1 \\ 1 & 0 \end{pmatrix},\ 
\rho\left(x_{6k+5}\right) = \begin{pmatrix} 1 & 0 \\ 1 & 1 \end{pmatrix}
\end{eqnarray*}
if $n=6k+5\ (k\ge 0)$, and 
\begin{eqnarray*}
\rho\left(x_{i}\right) &=& \begin{pmatrix} 0 & 1 \\ 1 & 1 \end{pmatrix}\ (i=1,2,\ldots,6k+3),\\
\rho\left(x_{6k+4}\right) &=& \begin{pmatrix} 0 & 1 \\ 1 & 0 \end{pmatrix},\ 
\rho\left(x_{6k+5}\right) = \begin{pmatrix} 1 & 0 \\ 1 & 1 \end{pmatrix},\\
\rho\left(x_{6k+6}\right) &=& \begin{pmatrix} 0 & 1 \\ 1 & 0 \end{pmatrix},\ 
\rho\left(x_{6k+7}\right) = \begin{pmatrix} 1 & 0 \\ 1 & 1 \end{pmatrix}
\end{eqnarray*}
if $n=6k+7\ (k\ge 0)$.
Then $\rho$ is a group representation.
\end{Lemma}

\begin{proof}
By Lemma \ref{Su}, it is sufficient to show that $\rho\left(r\right)$ equals to the identity matrix $E$.
In the case of $n=6k+5\ (k\ge 0)$, we have
\begin{eqnarray}
\rho\left(x_{1}x_{6k+5}x_{1}^{-1}\right)
&=& \begin{pmatrix} 0 & 1 \\ 1 & 0 \end{pmatrix},\label{65_1}
\end{eqnarray}
\begin{eqnarray}
\rho\left(x_{i+1}x_{i}x_{i+1}^{-1}\right)
 &=& \begin{pmatrix} 0 & 1 \\ 1 & 1 \end{pmatrix}\label{65_2}\ (i=1,2,\ldots,6k+2),
\end{eqnarray}
\begin{eqnarray}
\rho\left(x_{6k+4}x_{6k+3}x_{6k+4}^{-1}\right)
 &=& \begin{pmatrix} 1 & 1 \\ 1 & 0 \end{pmatrix},\label{65_3}
\end{eqnarray}
\begin{eqnarray}
\rho\left(x_{6k+5}x_{6k+4}x_{6k+5}^{-1}\right)
 &=& \begin{pmatrix} 1 & 1 \\ 0 & 1 \end{pmatrix}.\label{65_4}
\end{eqnarray}
Note that 
\begin{eqnarray*}
{\begin{pmatrix} 0 & 1 \\ 1 & 1 \end{pmatrix}}^{3} = E
\end{eqnarray*}
in $SL\left(2;{\mathbb Z}_{2}\right)$. Then by (\ref{65_1}), (\ref{65_2}), (\ref{65_3}) and (\ref{65_4}), we have 
\begin{eqnarray*}
\rho\left(r\right)
&=& \rho\left(x_{1}x_{6k+5}x_{1}^{-1}\right)
\rho\left(x_{2}x_{1}x_{2}^{-1}\right)\cdots \rho\left(x_{6k+5}x_{6k+4}x_{6k+5}^{-1}\right)\\
&=& 
\begin{pmatrix} 0 & 1 \\ 1 & 0 \end{pmatrix}
{\begin{pmatrix} 0 & 1 \\ 1 & 1 \end{pmatrix}}^{2}
\begin{pmatrix} 1 & 1 \\ 1 & 0 \end{pmatrix}
\begin{pmatrix} 1 & 1 \\ 0 & 1 \end{pmatrix}\\
&=& E.
\end{eqnarray*}
In the case of $n=6k+7\ (k\ge 0)$, we have 
\begin{eqnarray}
\rho\left(x_{1}x_{6k+7}x_{1}^{-1}\right)
 &=& \begin{pmatrix} 0 & 1 \\ 1 & 0 \end{pmatrix},\label{67_1}
\end{eqnarray}
\begin{eqnarray}
\rho\left(x_{i+1}x_{i}x_{i+1}^{-1}\right)
 &=& \begin{pmatrix} 0 & 1 \\ 1 & 1 \end{pmatrix}\label{67_2}\ (i=1,2,\ldots,6k+2),
\end{eqnarray}
\begin{eqnarray}
 \rho\left(x_{6k+4}x_{6k+3}x_{6k+4}^{-1}\right)
 &=& \begin{pmatrix} 1 & 1 \\ 1 & 0 \end{pmatrix},\label{67_3}
\end{eqnarray}
\begin{eqnarray} 
 \rho\left(x_{6k+j}x_{6k+j-1}x_{6k+j}^{-1}\right)
 &=& \begin{pmatrix} 1 & 1 \\ 0 & 1 \end{pmatrix}\ (j=5,6,7).\label{67_4}
\end{eqnarray}
Then by (\ref{67_1}), (\ref{67_2}), (\ref{67_3}) and (\ref{67_4}), we have 
\begin{eqnarray*}
\rho\left(r\right)
&=& \rho\left(x_{1}x_{6k+7}x_{1}^{-1}\right)
\rho\left(x_{2}x_{1}x_{2}^{-1}\right)\cdots \rho\left(x_{6k+7}x_{6k+6}x_{6k+7}^{-1}\right)\\
&=& 
\begin{pmatrix} 0 & 1 \\ 1 & 0 \end{pmatrix}
{\begin{pmatrix} 0 & 1 \\ 1 & 1 \end{pmatrix}}^{2}
\begin{pmatrix} 1 & 1 \\ 1 & 0 \end{pmatrix}
{\begin{pmatrix} 1 & 1 \\ 0 & 1 \end{pmatrix}}^{3}\\
&=& E.
\end{eqnarray*}
Thus we have the desired conclusion.
\end{proof}

Now we state our main theorem in this section.

\begin{Theorem}\label{suzuki_twist}
Let $n$ be a positive integer satisfying $n\ge 3$ and $n\equiv 1,5\pmod{6}$. Let $\rho$ be a group presentation from $G_{n}$ to $SL\left(2;{\mathbb Z}_{2}\right)$ defined in Lemma \ref{rho_homo}. Then it follows that 
\begin{eqnarray*}
E_{d}\left(A\left(G_{n}, \tilde{\rho}\otimes\tilde{\alpha}_{n}\right)\right)=\left\{
   \begin{array}{@{\,}ll}
       \left(0\right) & \mbox{$(d<2n-2)$} \\
       \left(1+t\right) & \mbox{$(d=2n-2)$} \\
       \left(1\right) & \mbox{$(d \ge 2n-1)$}
   \end{array}
   \right. . 
\end{eqnarray*}
\end{Theorem}

\begin{proof}
We denote the composition map $\left(\tilde{\rho}\otimes\tilde{\alpha}_{n}\right)\circ\tilde{\phi}$ by $\Phi_{n}$. By (\ref{fd1}) and (\ref{fd2}), we have
\begin{eqnarray}\label{afd1}
 \Phi_{n}
\left(\frac{\partial r}{\partial x_{i}}\right)
&=& 
\left\{\prod_{m=1}^{i-1}
\Phi_{n}\left(x_{m}x_{m-1}x_{m}^{-1}\right)\right\}\\
&& \cdot\left\{E-\Phi_{n}\left(x_{i}x_{i-1}x_{i}^{-1}\right)
+\Phi_{n}\left(x_{i}x_{i-1}x_{i}^{-1}\right)
\cdot\Phi_{n}\left(x_{i+1}\right)\right\}\nonumber
\end{eqnarray}
for $i=1,2,\ldots,n-1$ and 
\begin{eqnarray}\label{afd2}
\Phi_{n}
\left(\frac{\partial r}{\partial x_{n}}\right)&=& 
\Phi_{n}\left(x_{1}\right)
+\prod_{m=1}^{n-1}\Phi_{n}\left(x_{m}x_{m-1}x_{m}^{-1}\right)-E. 
\end{eqnarray}

First we show in the case of $n=6k+5$. By combining (\ref{65_1}), (\ref{65_2}), (\ref{65_3}) (\ref{65_4}) with (\ref{afd1}), (\ref{afd2}), we have 
\begin{eqnarray}
\Phi_{n}\left(\frac{\partial r}{\partial x_{1}}\right)
&=& \begin{pmatrix} 1 & 0 \\ 0 & 1 \end{pmatrix}
-t^{-6k-4}\begin{pmatrix} 0 & 1 \\ 1 & 0 \end{pmatrix}
+t^{-6k-3}\begin{pmatrix} 1 & 1 \\ 0 & 1 \end{pmatrix}\label{65a1}\\
&=& t^{-6k-4}{\begin{pmatrix} t+t^{6k+4} & 1+t \\ 1 & t+ t^{6k+4} \end{pmatrix}},\nonumber
\end{eqnarray}
\begin{eqnarray}
&&\Phi_{n}\left(\frac{\partial r}{\partial x_{3l+2}}\right)\label{65a2}\\
&=& 
t^{3l}\left\{t^{-6k-4}\begin{pmatrix} 0 & 1 \\ 1 & 0 \end{pmatrix}
-t^{-6k-3}\begin{pmatrix} 1 & 1 \\ 0 & 1 \end{pmatrix}
+t^{-6k-2}\begin{pmatrix} 1 & 0 \\ 1 & 1 \end{pmatrix}
\right\}\nonumber
\\
&=& t^{3l-6k-4}{\begin{pmatrix} t+t^{2} & 1+t \\ 1+t^{2} & t+t^{2} \end{pmatrix}}\ (l=0,1,\ldots,2k-1),\nonumber
\end{eqnarray}
\begin{eqnarray}
&&\Phi_{n}\left(\frac{\partial r}{\partial x_{3l+3}}\right)\label{65a3}\\
&=& 
t^{3l}\left\{t^{-6k-3}\begin{pmatrix} 1 & 1 \\ 0 & 1 \end{pmatrix}
- t^{-6k-2}\begin{pmatrix} 1 & 0 \\ 1 & 1 \end{pmatrix}
+ t^{-6k-1}\begin{pmatrix} 0 & 1 \\ 1 & 0 \end{pmatrix}
\right\}\nonumber
\\
&=& t^{3l-6k-3}{\begin{pmatrix} 1+t & 1+t^{2} \\ t+t^{2} & 1+t \end{pmatrix}}\ (l=0,1,\ldots,2k-1),\nonumber
\end{eqnarray}
\begin{eqnarray}
&&\Phi_{n}\left(\frac{\partial r}{\partial x_{3l+4}}\right)\label{65a4}\\
&=& 
t^{3l}\left\{t^{-6k-2}\begin{pmatrix} 1 & 0 \\ 1 & 1 \end{pmatrix}
- t^{-6k-1}\begin{pmatrix} 0 & 1 \\ 1 & 0 \end{pmatrix}
+ t^{-6k}\begin{pmatrix} 1 & 1 \\ 0 & 1 \end{pmatrix}
\right\}\nonumber
\\
&=& t^{3l-6k-2}{\begin{pmatrix} 1+t^{2} & t+t^{2} \\ 1+t & 1+t^{2} \end{pmatrix}}\ (l=0,1,\ldots,2k-1),\nonumber
\end{eqnarray}
\begin{eqnarray}
\Phi_{n}\left(\frac{\partial r}{\partial x_{6k+2}}\right)
&=& 
t^{-4}\begin{pmatrix} 0 & 1 \\ 1 & 0 \end{pmatrix}
- t^{-3}\begin{pmatrix} 1 & 1 \\ 0 & 1 \end{pmatrix}
+ t^{-2}\begin{pmatrix} 1 & 0 \\ 1 & 1 \end{pmatrix}
\label{65a5}\\
&=& t^{-4}{\begin{pmatrix} t+t^{2} & 1+t \\ 1+t^{2} & t+t^{2} \end{pmatrix}},\nonumber
\end{eqnarray}
\begin{eqnarray}
\Phi_{n}\left(\frac{\partial r}{\partial x_{6k+3}}\right)
&=& 
t^{-3}\begin{pmatrix} 1 & 1 \\ 0 & 1 \end{pmatrix}
- t^{-2}\begin{pmatrix} 1 & 0 \\ 1 & 1 \end{pmatrix}
+ t^{-1}\begin{pmatrix} 0 & 1 \\ 1 & 1 \end{pmatrix}
\label{65a6}\\
&=& t^{-3}{\begin{pmatrix} 1+t & 1+t^{2} \\ t+t^{2} & 1+t+t^{2} \end{pmatrix}},\nonumber
\end{eqnarray}
\begin{eqnarray}
\Phi_{n}\left(\frac{\partial r}{\partial x_{6k+4}}\right)
&=& 
t^{-2}\begin{pmatrix} 1 & 0 \\ 1 & 1 \end{pmatrix}
- t^{-1}\begin{pmatrix} 1 & 1 \\ 0 & 1 \end{pmatrix}
+ t^{-6k-5}\begin{pmatrix} 0 & 1 \\ 1 & 1 \end{pmatrix}
\label{65a7}\\
&=& t^{-6k-5}{\begin{pmatrix} t^{6k+3}+t^{6k+4} & 1+t^{6k+4} \\ 1+t^{6k+3} & 1+t^{6k+3}+t^{6k+4} \end{pmatrix}},\nonumber
\end{eqnarray}
\begin{eqnarray}
\Phi_{n}\left(\frac{\partial r}{\partial x_{6k+5}}\right)
&=& 
t^{-1}\begin{pmatrix} 1 & 1 \\ 0 & 1 \end{pmatrix}
- \begin{pmatrix} 1 & 0 \\ 0 & 1 \end{pmatrix}
+ t\begin{pmatrix} 0 & 1 \\ 1 & 1 \end{pmatrix}
\label{65a8}\\
&=& t^{-1}{\begin{pmatrix} 1+t & 1+t^{2} \\ t^{2} & 1+t+t^{2} \end{pmatrix}}.\nonumber
\end{eqnarray}

By (\ref{65a2}) and (\ref{65a8}), we see that 
\begin{eqnarray}\label{2-n}
&&\left(\Phi_{n}\left(\frac{\partial r}{\partial x_{2}}\right)\ \Phi_{n}\left(\frac{\partial r}{\partial x_{6k+5}}\right)\right)\\
&\sim& 
\begin{pmatrix} 
t+t^{2} & 1+t & 1+t & 1+t^{2} \\ 
1+t^{2} & t+t^{2} & t^{2} & 1+t+t^{2} 
\end{pmatrix}\nonumber\\
&\sim& 
\begin{pmatrix} 
t+t^{2} & 0 & 1+t & 1+t^{2} \\ 
1+t^{2} & t & t^{2} & 1+t+t^{2} 
\end{pmatrix}\nonumber\\
&\sim& 
\begin{pmatrix} 
t+t^{2} & 0 & 1+t & 1+t^{2} \\ 
0 & 1 & 0 & 0 
\end{pmatrix}\nonumber\\
&\sim& 
\begin{pmatrix} 
0 & 0 & 1+t & 0 \\ 
0 & 1 & 0 & 0 
\end{pmatrix}\nonumber
\end{eqnarray}
by the elementary transformations of columns. Then by (\ref{65a1})--(\ref{65a8}) and (\ref{2-n}), it is easy to see that
\begin{eqnarray*}
\left(\Phi_{n}\left(\frac{\partial r}{\partial x_{i}}\right)\right)_{i=1,2,\ldots,6k+5}
\sim 
\begin{pmatrix} 
1+t & 0 & 0 & \cdots & 0\\ 
0 & 1  & 0 & \cdots & 0
\end{pmatrix}.
\end{eqnarray*}
This implies the desired conclusion. 

Next we show in the case of $n=6k+7$. By combining (\ref{67_1}), (\ref{67_2}), (\ref{67_3}) (\ref{67_4}) with (\ref{afd1}), (\ref{afd2}), we have 
\begin{eqnarray}
\Phi_{n}\left(\frac{\partial r}{\partial x_{1}}\right)
&=& \begin{pmatrix} 1 & 0 \\ 0 & 1 \end{pmatrix}
- t^{-6k-6}\begin{pmatrix} 0 & 1 \\ 1 & 0 \end{pmatrix}
+ t^{-6k-5}\begin{pmatrix} 1 & 1 \\ 0 & 1 \end{pmatrix}\label{67a1}\\
&=& t^{-6k-6}{\begin{pmatrix} t+t^{6k+6} & 1+t \\ 1 & t+ t^{6k+6} \end{pmatrix}},\nonumber
\end{eqnarray}
\begin{eqnarray}
&&\Phi_{n}\left(\frac{\partial r}{\partial x_{3l+2}}\right)\label{67a2}\\
&=& 
t^{3l}\left\{t^{-6k-6}\begin{pmatrix} 0 & 1 \\ 1 & 0 \end{pmatrix}
- t^{-6k-5}\begin{pmatrix} 1 & 1 \\ 0 & 1 \end{pmatrix}
+ t^{-6k-4}\begin{pmatrix} 1 & 0 \\ 1 & 1 \end{pmatrix}
\right\}\nonumber
\\
&=& t^{3l-6k-6}{\begin{pmatrix} t+t^{2} & 1+t \\ 1+t^{2} & t+t^{2} \end{pmatrix}}\ (l=0,1,\ldots,2k-1),\nonumber
\end{eqnarray}
\begin{eqnarray}
&&\Phi_{n}
\left(\frac{\partial r}{\partial x_{3l+3}}\right)\label{67a3}\\
&=& 
t^{3l}\left\{t^{-6k-5}\begin{pmatrix} 1 & 1 \\ 0 & 1 \end{pmatrix}
- t^{-6k-4}\begin{pmatrix} 1 & 0 \\ 1 & 1 \end{pmatrix}
+ t^{-6k-3}\begin{pmatrix} 0 & 1 \\ 1 & 0 \end{pmatrix}
\right\}\nonumber
\\
&=& t^{3l-6k-5}{\begin{pmatrix} 1+t & 1+t^{2} \\ t+t^{2} & 1+t \end{pmatrix}}\ (l=0,1,\ldots,2k-1),\nonumber
\end{eqnarray}
\begin{eqnarray}
&&\Phi_{n}
\left(\frac{\partial r}{\partial x_{3l+4}}\right)\label{67a4}\\
&=& 
t^{3l}\left\{t^{-6k-4}\begin{pmatrix} 1 & 0 \\ 1 & 1 \end{pmatrix}
- t^{-6k-3}\begin{pmatrix} 0 & 1 \\ 1 & 0 \end{pmatrix}
+ t^{-6k-2}\begin{pmatrix} 1 & 1 \\ 0 & 1 \end{pmatrix}
\right\}\nonumber
\\
&=& t^{3l-6k-4}{\begin{pmatrix} 1+t^{2} & t+t^{2} \\ 1+t & 1+t^{2} \end{pmatrix}}\ (l=0,1,\ldots,2k-1),\nonumber
\end{eqnarray}
\begin{eqnarray}
\Phi_{n}
\left(\frac{\partial r}{\partial x_{6k+2}}\right)
&=& 
t^{-6}\begin{pmatrix} 0 & 1 \\ 1 & 0 \end{pmatrix}
- t^{-5}\begin{pmatrix} 1 & 1 \\ 0 & 1 \end{pmatrix}
+ t^{-4}\begin{pmatrix} 1 & 0 \\ 1 & 1 \end{pmatrix}
\label{67a5}\\
&=& t^{-6}{\begin{pmatrix} t+t^{2} & 1+t \\ 1+t^{2} & t+t^{2} \end{pmatrix}},\nonumber
\end{eqnarray}
\begin{eqnarray}
\Phi_{n}
\left(\frac{\partial r}{\partial x_{6k+3}}\right)
&=& 
t^{-5}\begin{pmatrix} 1 & 1 \\ 0 & 1 \end{pmatrix}
- t^{-4}\begin{pmatrix} 1 & 0 \\ 1 & 1 \end{pmatrix}
+ t^{-3}\begin{pmatrix} 0 & 1 \\ 1 & 1 \end{pmatrix}
\label{67a6}\\
&=& t^{-5}{\begin{pmatrix} 1+t & 1+t^{2} \\ t+t^{2} & 1+t+t^{2} \end{pmatrix}},\nonumber
\end{eqnarray}
\begin{eqnarray}
\Phi_{n}
\left(\frac{\partial r}{\partial x_{6k+4}}\right)
&=& 
t^{-4}\begin{pmatrix} 1 & 0 \\ 1 & 1 \end{pmatrix}
- t^{-3}\begin{pmatrix} 1 & 1 \\ 0 & 1 \end{pmatrix}
+ t^{-2}\begin{pmatrix} 0 & 1 \\ 1 & 1 \end{pmatrix}
\label{67a7}\\
&=& t^{-4}{\begin{pmatrix} 1+t & t+t^{2} \\ 1+t^{2} & 1+t+t^{2} \end{pmatrix}},\nonumber
\end{eqnarray}
\begin{eqnarray}
\Phi_{n}
\left(\frac{\partial r}{\partial x_{6k+5}}\right)
&=& 
t^{-3}\begin{pmatrix} 1 & 1 \\ 0 & 1 \end{pmatrix}
- t^{-2}\begin{pmatrix} 1 & 0 \\ 0 & 1 \end{pmatrix}
+ t^{-1}\begin{pmatrix} 0 & 1 \\ 1 & 0 \end{pmatrix}
\label{67a8}\\
&=& t^{-3}{\begin{pmatrix} 1+t & 1+t^{2} \\ t^{2} & 1+t \end{pmatrix}},\nonumber
\end{eqnarray}
\begin{eqnarray}
\Phi_{n}
\left(\frac{\partial r}{\partial x_{6k+6}}\right)
&=&
t^{-2}\begin{pmatrix} 1 & 0 \\ 0 & 1 \end{pmatrix}
- t^{-1}\begin{pmatrix} 1 & 1 \\ 0 & 1 \end{pmatrix}
+ t^{-6k-7}\begin{pmatrix} 0 & 1 \\ 1 & 1 \end{pmatrix}
\label{67a9}\\
&=& t^{-6k-7}{\begin{pmatrix} t^{6k+5}+t^{6k+6} & 1+t^{6k+5} \\ 1 & 1+t^{6k+5}+t^{6k+6} \end{pmatrix}},\nonumber
\end{eqnarray}
\begin{eqnarray}
\Phi_{n}
\left(\frac{\partial r}{\partial x_{6k+7}}\right)
&=& 
t^{-1}\begin{pmatrix} 1 & 1 \\ 0 & 1 \end{pmatrix}
- \begin{pmatrix} 1 & 0 \\ 0 & 1 \end{pmatrix}
+ t\begin{pmatrix} 0 & 1 \\ 1 & 1 \end{pmatrix}
\label{67a10}\\
&=& t^{-1}{\begin{pmatrix} 1+t & 1+t^{2} \\ t^{2} & 1+t+t^{2} \end{pmatrix}}.\nonumber
\end{eqnarray}
By (\ref{67a2}) and (\ref{67a10}), we see that 
\begin{eqnarray}\label{2-n_2}
\left(\Phi_{n}\left(\frac{\partial r}{\partial x_{2}}\right)\ \Phi_{n}\left(\frac{\partial r}{\partial x_{6k+7}}\right)\right)
&\sim& 
\begin{pmatrix} 
0 & 0 & 1+t & 0 \\ 
0 & 1 & 0 & 0 
\end{pmatrix}
\end{eqnarray}
in the same way as (\ref{2-n}). Then by (\ref{67a1})--(\ref{67a10}) and (\ref{2-n_2}), it is easy to see that
\begin{eqnarray*}
A\left(G_{n},\tilde{\rho}\otimes\tilde{\alpha}_{n}\right)\sim
\begin{pmatrix} 
1+t & 0 & 0 & \cdots & 0\\ 
0 & 1  & 0 & \cdots & 0
\end{pmatrix}.
\end{eqnarray*}
This implies the desired conclusion. 
\end{proof}

By Lemma \ref{trivia} and Theorem \ref{suzuki_twist}, we succeed to show the nontriviality of $\Theta_{n}$ for $n\equiv 1,5\pmod{6}$ by using twisted Alexander ideals.

\section{Handlebody-knots}

A \textit{handlebody-knot} is a handlebody embedded in $S^3$. Two handlebody-knots are said to be \textit{equivalent} if there exists an orientation-preserving self-homeomorphism on $S^3$ which sends one to the other. A \textit{diagram} of a handlebody-knot $H$ is that of a spatial trivalent graph $\Gamma_{H}$ whose regular neighborhood is equivalent to $H$ as a handlebody-knot. A table of genus two handlebody-knots up to six crossings was given in \cite{IshiiKishimotoMoriuchiSuzuki12}. In this section, we evaluate the twisted Alexander ideals for the handlebody-knots in the table. 

For a handlebody-knot $H$, we denote the fundamental group of the exterior of $H$ in $S^{3}$ by $G_{H}$. Since the exterior of $H$ and the exterior of $\Gamma_{H}$ in $S^{3}$ are homeomorphic, it follows that $G_{H}\cong G_{\Gamma_{H}}$. Therefore we can obtain a presentation of $G_{H}$ by taking the Wirtinger presentation for a diagram of $H$. Moreover, the (twisted) Alexander ideals of $H$ is derived from $G_H$ with (a group representation $\rho$ from $G_{H}$ to a matrix group $G$ and) an epimorhism $\alpha$ from $G_H$ to an abelian group $G_0$ as in the case of spatial graphs. Unlike the case of spatial graphs, we cannot specify $\rho$ and $\alpha$, since there is no canonical meridian system for handlebody-knots. To ensure the invariance, we sum up the (twisted) Alexander ideals over all possible ($\rho$ and) $\alpha$. Namely, the collection of the (twisted) Alexander ideals of $G_{H}$ associated with all possible ($\rho$ and) $\alpha$ is an invariant of $H$. For groups $G_{1}$ and $G_{2}$, we denote by $\operatorname{Conj}\left(G_{1},G_{2}\right)$ the set of representative elements of conjugacy classes of homomorphisms from $G_{1}$ to $G_{2}$, and by $\operatorname{Epi}\left(G_{1},G_{2}\right)$ the set of all epimorphisms from $G_{1}$ to $G_{2}$. Then we obtain a handlebody-knot invariant of matrix form 
\[ \left(E_{d}\left(A\left(G_{H},\tilde{\rho}\otimes\tilde{\alpha}\right)\right)\right)_{\rho\in\operatorname{Conj}(G_H,G),\alpha\in\operatorname{Epi}(G_H,G_{0})}, \]
where two matrices are assumed to be the same if one can be transformed to the other by permuting rows and columns.
Set $G = SL\left(2;\mathbb{Z}_2\right)$, $G_0=\langle t\ |\ t^2\rangle$ and $d=4$.
Table \ref{table:data_for_handlebody-knot} lists the invariant of matrix form, where 
\begin{eqnarray*}
\left\{\left(a_{11},a_{12}\ldots,a_{1n}\right)_{l_1},\left(a_{21},a_{22}\ldots,a_{2n}\right)_{l_2},\ldots,\left(a_{m1},a_{m2},\ldots,a_{mn}\right)_{l_m}\right\}
\end{eqnarray*}
indicates the matrix
\[ \begin{pmatrix}
\left.
\begin{array}{cccc}
\left(a_{11}\right) & \left(a_{12}\right) & \ldots & \left(a_{1n}\right) \\
\vdots & \vdots & & \vdots \\
\left(a_{11}\right) & \left(a_{12}\right) & \ldots & \left(a_{1n}\right)
\end{array}
\right\}\text{\scriptsize$l_1$} \\
\vdots\hspace{7mm} \\
\left.
\begin{array}{cccc}
\left(a_{m1}\right) & \left(a_{m2}\right) & \ldots & \left(a_{mn}\right) \\
\vdots & \vdots & & \vdots \\
\left(a_{m1}\right) & \left(a_{m2}\right) & \ldots & \left(a_{mn}\right)
\end{array}
\right\}\text{\scriptsize$l_m$}
\end{pmatrix}. \]
For example,
\[ \{(1,1,1)_2,(0,t+1,t+1)_1,(t+1,0,t+1)_1\}
=\begin{pmatrix}
(1) & (1) & (1) \\
(1) & (1) & (1) \\
(0) & (t+1) & (t+1) \\
(t+1) & (0) & (t+1)
\end{pmatrix}. \]

The second column of Table \ref{table:data_for_handlebody-knot} shows the number of the conjugacy classes of representations of $G_H$ on $SL(2;\mathbb{Z}_2)$ \cite{IshiiKishimotoMoriuchiSuzuki12}. Then, from the table, we see that the invariant of matrix form works better than the number of the conjugacy classes of representations. Although counting representations is easy way to distinguish two handlebody-knots, the evaluation is not easy if the representation space is big. In such case, the invariant discussed in this section may work well with a small representation space. 

\begin{table}[htbp]
\begin{center}
\scalebox{0.8}{
\begin{tabular}{|c||c|l|}
\hline
$H$ & &
\\ \hline
$0_1$ & $11$ & $\{(1,1,1)_{11}\}$
\\ \hline
$4_1$ & $14$ & $\{(0,0,0)_1,(1,1,1)_{10},(0,1+t,1+t)_1,(1+t,0,1+t)_1,(1+t,1+t,0)_1\}$
\\ \hline
$5_1$ & $11$ & $\{(1,1,1)_9,(0,1+t,1+t)_1,(1+t,0,1+t)_1\}$
\\ \hline
$5_2$ & $14$ & $\{(0,0,0)_1,(1,1,1)_{10},(0,1+t,1+t)_1,(1+t,0,1+t)_1,(1+t,1+t,0)_1\}$
\\ \hline
$5_3$ & $11$ & $\{(0,0,0)_1,(1,1,1)_{10}\}$
\\ \hline
$5_4$ & $11$ & $\{(0,0,0)_3,(1,1,1)_8\}$
\\ \hline
$6_1$ & $11$ & $\{(0,0,0)_2,(1,1,1)_9\}$
\\ \hline
$6_2$ & $11$ & $\{(0,0,0)_1,(1,1,1)_{10}\}$
\\ \hline
$6_3$ & $11$ & $\{(1,1,1)_{11}\}$
\\ \hline
$6_4$ & $11$ & $\{(1,1,1)_9,(0,1+t,1+t)_1,(1+t,0,1+t)_1\}$
\\ \hline
$6_5$ & $11$ & $\{(1,1,1)_{10},(0,1+t,1+t)_1\}$
\\ \hline
$6_6$ & $11$ & $\{(1,1,1)_{10},(0,1+t,1+t)_1\}$
\\ \hline
$6_7$ & $11$ & $\{(0,0,0)_1,(1,1,1)_{10}\}$
\\ \hline
$6_8$ & $11$ & $\{(1,1,1)_{11}\}$
\\ \hline
$6_9$ & $14$ & $\{(0,0,0)_3,(1,1,1)_{10},(0,1+t,1+t)_1\}$
\\ \hline
$6_{10}$ & $11$ & $\{(1,1,1)_{11}\}$
\\ \hline
$6_{11}$ & $11$ & $\{(1,1,1)_{11}\}$
\\ \hline
$6_{12}$ & $11$ & $\{(0,0,0)_1,(1,1,1)_9,(0,1+t,1+t)_1\}$
\\ \hline
$6_{13}$ & $14$ & $\{(0,0,0)_1,(1,1,1)_{10},(0,1+t,1+t)_1,(1+t,0,1+t)_1,(1+t,1+t,0)_1\}$
\\ \hline
$6_{14}$ & $17$ & $\{(0,0,0)_9,(1,1,1)_8\}$
\\ \hline
$6_{15}$ & $17$ & $\{(0,0,0)_9,(1,1,1)_8\}$
\\ \hline
$6_{16}$ & $11$ & $\{(0,0,0)_3,(1,1,1)_8\}$
\\ \hline
\end{tabular}
}
\end{center}
\vspace{0.2cm}
\caption{}
\label{table:data_for_handlebody-knot}
\end{table}

\section{surface-links}

A {\it surface-link} is a closed surface locally flatly embedded in $\mathbb{R}^4$. Two surface-links are said to be {\it equivalent} if there exists an orientation-preserving self-homeomorphism on $\mathbb{R}^4$ which sends one to the other. It is well-known that any surface-link can be deformed into a surface-link which has a Morse position with respect to the fourth coordinate and whose maximal and minimal points are in the hyperplane $\mathbb{R}^3\times\{1\}$ and $\mathbb{R}^3\times \{-1\}$, respectively, and hyperbolic points are in the hyperplane $\mathbb{R}^3\times\{0\}$ (cf. \cite{FoxMilnor66, Kamada89, KawauchiShibuyaSuzuki82}). 
The $0$-level cross-section with a ``marking'' of each vertex is a $4$-valent graph which realizes the original surface-link. The diagram is called a {\it ch-diagram} of the surface-link. 
In \cite{Yoshikawa94}, Yoshikawa gave a complete list of surface-links which have ch-diagrams such that the sum of the number of crossings and that of hyperbolic vertices is less than or equal to $10$. The {\it knot group} $G_{F}$ of a surface-link $F$, 
that is the fundamental group of the complement of $F$, can be calculated by using a Wirtinger presentation of $F$, refer to \cite{Fox61} for the computation from motion pictures of surface-links and \cite{CKS04} for the computation from ch-diagrams, and the (twisted) Alexander ideals of $F$ is derived from $G_{F}$ with (a group representation $\rho$ from $G_{F}$ to a matrix group $G$ and) an epimorphism $\alpha$ from $G_H$ to an abelian group $G_0$. 
Table \ref{table:Yoshikawa_data_for_2-link} is Yoshikawa's original table equipped with the information about the knot groups and the first Alexander ideals associated with the abelianizers. 
We correct three mistakes: the knot group of $9_1^{1,-2}$, its first Alexander ideal and the first Alexander ideal of $10_2^{0,-2}$. 
Let $\alpha$ be the homomorphism from $G_{F}$ to $G_{0}$ which sends each Wirtinger generator to $t$. Then 
we obtain a surface-link invariant of matrix form
\begin{eqnarray*}
\left(E_d\left(A\left(G_{F},\tilde{\rho}\otimes\tilde{\alpha}\right)
\right)\right)_{\rho\in\operatorname{Conj}\left(G_H,G\right), d=1,2,\ldots},
\end{eqnarray*}
where two matrices are assumed to be the same if one can be transformed to the other by permuting rows.
Set $G=SL\left(2;\mathbb{Z}_2\right)$ and $G_0=\langle t\ |\ t^2\rangle$. 
Table \ref{table:data_for_2-link} lists this invariant for the surface-links in Yoshikawa's table, where we omit the last consecutive 1's in each pair of parentheses, that is,   
\[
\{ \left(a_{11}, a_{12}, \ldots , a_{1 n_{1}-1},1\right)_{l_1}, \ldots , \left(a_{m1}, a_{m2}, \ldots , a_{m n_{m}-1},1\right)_{l_m}\}
\]
represents 
\[
\{ \left(a_{11}, a_{12},  \ldots , a_{1 n_{1}-1},1, 1, 1, \ldots \right)_{l_1} , \ldots , \left(a_{m1}, a_{m2}, \ldots , a_{m n_m-1},1, 1, 1, \ldots \right)_{l_m}\}
\]
and indicates the matrix with infinite columns as shown in \S4.

Consider the cases of $8_1^{-1,-1}$ and $9_{1}^{1,-2}$.
They have the same first Alexander ideals. 
However they are distinguished 
by comparing the second twisted Alexander ideals of $8_1^{-1,-1}$ for all the group representations to $SL\left(2;\mathbb{Z}_2\right)$
 and that of $9_1^{1,-2}$ for a group representation to $SL\left(2;\mathbb{Z}_2\right)$.
Thus $8_1^{-1,-1}$ and $9_{1}^{1,-2}$ gives an example of  two surface-links which are not distinguished by the first Alexander ideals but distinguished by the twisted Alexander ideals.
Note that they are also distinguished by the $0$th Alexander ideals or the number of group representations to $SL\left(2;\mathbb{Z}_2\right)$, and moreover, the surface-links are clearly different since $8_1^{-1,-1}$ does not have an orientable component but $9_{1}^{1,-2}$ does.

\begin{Remark}
When we have an example of two classical-knots which   are not distinguished by the Alexander ideals but distinguished by twisted Alexander ideals, we can easily construct an example of surface-links which satisfies the same property  by taking the Artin's spinning process \cite{Artin26} of the classical-knots.
It is known that the Kinoshita-Terasaka knot $K_{\rm KT}$ and the Conway knot $K_{\rm C}$ are not distinguished by the Alexander ideals but distinguished by  twisted Alexander ideals with the parabolic representations to $SL\left(2;\mathbb{Z}_7\right)$, see \cite{Wada94}.
Thus the spun $K_{\rm KT}$ and the spun $K_{\rm C}$ gives an example of  two surface-links which are not distinguished by the Alexander ideals but distinguished by twisted Alexander ideals.
\end{Remark}

\begin{table}[htbp]
\begin{center}
\scalebox{0.8}{
\begin{tabular}{|c||l|l|l|}
\hline
$F$ & $\pi(\mathbb R^4 -F)$ & ideal & polynomial  
\\ \hline
$0_1$ & $\mathbb Z$ & $(1)$ & $1$
\\ \hline
$2^1_1$ & $\mathbb Z$ & $(1)$ & $1$
\\ \hline
$2^{-1}_1$ & $\mathbb Z_2$ & $(1)$ & $1$
\\ \hline
$6^{0,1}_1$ & $\mathbb Z\oplus \mathbb Z$ & $(x-1, y-1)$ & $1$
\\ \hline
$7^{0,-2}_1$ & $\langle x, y \,|\, yxyx^{-1} \rangle$ & $(x+1, y-1)$ & $1$
\\ \hline
$8_1$ & $\langle x_1, x_2 \,|\, x_1 x_2 x_1 x_2^{-1} x_1^{-1} x_2^{-1} \rangle$ & $(x^2-x+1)$ & $x^2-x+1$  
\\ \hline
$8^{1,1}_1$ & $\mathbb Z \oplus \mathbb Z$ & $(x-1,y-1)$ & $1$  
\\ \hline
$8^{-1,-1}_1$ & $\langle x, y \,|\, xyxy^{-1}, x^{-2}y^2 \rangle$ & $(x+1, y+1, 2)$ & $1$ 
\\ \hline
$9_1$ & $\langle x_1, x_2 \,|\, x_1 x_2^{-1} x_1 x_2 x_1^{-1} x_2^{-1} \rangle$ & $(x-2)$ & $x-2$
\\ \hline
$9^{0,1}_1$ & $\langle x, y \,|\, x^{-1}y^{-1} xy x^{-1} yx y^{-1} \rangle$ & $((x-1)(y-1), (y-1)^2)$ & $y-1$
\\ \hline
$9^{1,-2}_1$ & $\langle x, y \,|\, xyxy^{-1}, x^2 \rangle$ & $(x+1, y+1,2 )$ & $1$ 
\\ \hline
$10_1$ & $\langle x_1, x_2 \,|\, x_1^{-1} x_2 x_1 x_2^{-1} x_1 x_2 x_1^{-1}x_2^{-1} x_1 x_2^{-1} \rangle$ & $(x^2-3x+1)$ & $x^2-3x+1$
\\ \hline
$10_2$ & $\langle x_1, x_2 \,|\, x_1 x_2 x_1 x_2^{-1} x_1^{-1} x_2^{-1},  x_1^{2} x_2 x_1^{-2} x_2^{-1} \rangle$ & $(x+1,3)$ & $1$
\\ \hline
$10_3$ & $\langle x_1, x_2 \,|\, x_1 x_2 x_1 x_2^{-1} x_1^{-1} x_2^{-1},  x_1^{3} x_2 x_1^{-3} x_2^{-1} \rangle$ & $(x^2+x+1,2)$ & $1$
\\ \hline
$10^1_1$ & $\langle x_1, x_2 \,|\, x_1 x_2 x_1 x_2^{-1} x_1^{-1} x_2^{-1} \rangle$ & $(x^2 -x +1)$ & $x^2 -x +1$ 
\\ \hline
$10^{0,1}_1$ & $\langle x,y \,|\, x^{-1} y^{-1} x^{-1} yxyxy^{-1}\rangle$ & $((x-1)(xy+1), (y-1)(xy+1))$& $xy+1$
\\ \hline
$10^{0,1}_2$ & $\langle x,y \,|\, x^2 y x^{-2} y^{-1}\rangle$ & $((x-1)(x+1),(y-1)(x+1))$ & $x+1$
\\ \hline
$10^{1,1}_1$ & $\mathbb Z \oplus \mathbb Z$ & $(x-1, y-1)$ & $1$
\\ \hline
$10^{0,0,1}_1$ & $\langle x,y,z \,|\, y^{-1} x^{-1} zxy z^{-1} \rangle$ & $(0)$ & $0$ 
\\ \hline
$10^{0,-2}_1$ & $\langle x,y \,|\, x^{-1} y^{-1} x y x^{-1} yxy \rangle$ & $(2x+y-1,4)$ & $1$
\\ \hline
$10^{0,-2}_2$ & $\langle x,y \,|\, xy^2 x^{-1} y^{-2}, y x^{-1} y^{-1} xy x^{-1} yx \rangle$ & $(2x+y-1, 4, 2x^2 +2)$  & $1$ 
\\ \hline
$10^{-1,-1}_1$ & $\langle x,y \,|\, x^2 y^2, yxyxyx^{-1} y^{-1} x^{-1} \rangle $ & $(x+1, y+1, 4)$ & $1$
\\ \hline
$10^{-2,-2}_1$ & $\langle x, y \,|\, xyxy^{-1}, x^{-2}y^2 \rangle$ & $(x+1, y+1, 2)$ & $1$ 
\\ \hline
\end{tabular}
}
\end{center}
\vspace{0.2cm}
\caption{}
\label{table:Yoshikawa_data_for_2-link}
\end{table}

\begin{table}
\begin{center}
\scalebox{0.8}{
\begin{tabular}{|c||l|}
\hline
$F$ & invariant
\\ \hline
$0_1$ & $\{(0,1)_3\}$
\\ \hline
$2^1_1$ & $\{(0,1)_3\}$
\\ \hline
$2^{-1}_1$ & $\{(1)_1,(1+t,1)_1\}$
\\ \hline
$6^{0,1}_1$ & $\{(0,1)_4,(0,0,1)_1,(0,1+t,1)_2,(0,0,1+t,1)_1\}$
\\ \hline
$7^{0,-2}_1$ & $\{(0,1)_2,(0,0,1)_1,(0,1+t,1)_2,(0,0,1+t,1)_1\}$
\\ \hline
$8_1$ & $\{(0,1)_2,(0,0,1)_1,(0,0,1+t,1)_1\}$
\\ \hline
$8^{1,1}_1$ & $\{(0,1)_4,(0,0,1)_1,(0,1+t,1)_2,(0,0,1+t,1)_1\}$
\\ \hline
$8^{-1,-1}_1$ & $\{(0,1)_3,(0,0,1+t,1)_1\}$
\\ \hline
$9_1$ & $\{(0,1)_4\}$
\\ \hline
$9^{0,1}_1$ & $\{(0,1)_4,(0,0,0,1)_2,(0,0,1+t,1)_3\}$
\\ \hline
$9^{1,-2}_1$ & $\{(0,1)_3,(0,1+t,1)_1,(0,0,1+t,1)_1\}$
\\ \hline
$10_1$ & $\{(0,1)_2,(0,0,1+t,1)_1\}$
\\ \hline
$10_2$ & $\{(0,1)_4\}$
\\ \hline
$10_3$ & $\{(0,1)_4\}$
\\ \hline
$10^1_1$ & $\{(0,1)_2,(0,0,1)_1,(0,0,1+t,1)_1\}$
\\ \hline
$10^{0,1}_1$ & $\{(0,1)_3,(0,0,1)_2,(0,0,0,1)_3,(0,0,1+t,1)_2\}$
\\ \hline
$10^{0,1}_2$ & $\{(0,1)_3,(0,0,1)_2,(0,0,0,1)_2,(0,0,1+t,1)_3\}$
\\ \hline
$10^{1,1}_1$ & $\{(0,1)_4,(0,0,1)_1,(0,1+t,1)_2,(0,0,1+t,1)_1\}$
\\ \hline
$10^{0,0,1}_1$ & $\{(0,0,0,1)_{16},(0,0,0,0,1)_4,(0,0,0,1+t,1)_8,(0,0,0,0,1+t,1)_3\}$
\\ \hline
$10^{0,-2}_1$ & $\{(0,0,1)_2,(0,0,1+t,1)_3\}$
\\ \hline
$10^{0,-2}_2$ & $\{(0,0,1)_2,(0,0,1+t,1)_3\}$
\\ \hline
$10^{-1,-1}_1$ & $\{(0,1)_1,(0,1+t,1)_2,(0,0,1+t,1)_1\}$
\\ \hline
$10^{-2,-2}_1$ & $\{(0,1)_3,(0,0,1+t,1)_1\}$
\\ \hline
\end{tabular}
}
\end{center}
\vspace{0.2cm}
\caption{}
\label{table:data_for_2-link}
\end{table}

\section*{Acknowledgment}

The authors are grateful to Professor Charles Livingston and Professor Kokoro Tanaka for their valuable comments.


\begin{thebibliography}{99}

\bibitem{Artin26}
E. Artin, 
Zur Isotopie zweidimensionaler Fl\"{a} chen im $\mathbb R^4$ (German), 
\textit{Abh. Math. Sem. Univ. Hamburg} \textbf{4} (1925), 174--177.

\bibitem{CKS04}
J. S. Carter, S. Kamada and M. Saito, 
Surfaces in $4$-space, 
Encyclopaedia of Mathematical Sciences, {\bf 142}, Low-Dimensional Topology, III, {\it Springer-Verlag, Berlin,} 2004.



\bibitem{CrowellFox77}
R. H. Crowell and R. H. Fox, Introduction to knot theory, 
Reprint of the 1963 original, Graduate Texts in Mathematics {\bf 57}, {\it Springer-Verlag, New York-Heidelberg,} 1977. 

\bibitem{Fox53}
R. H. Fox, Free differential calculus. I. Derivation in the free group ring, {\it Ann. of Math.} {\bf 57} (1953), 547--560.

\bibitem{Fox61}
R. H. Fox, 
A quick trip through knot theory, 
{\it Topology of $3$-manifolds and related topics (Proc. The Univ. of Georgia Institute, 1961),} 120--167, {\it Prentice-Hall, Englewood Cliffs, N.J.,} 1962. 


\bibitem{FoxMilnor66}
R. H. Fox and J. W. Milnor, 
Singularities of $2$-spheres in $4$-space and cobordism of knots, 
{\it Osaka J. Math.} \textbf{3} (1966), 257--267.

\bibitem{FriedlVidussi11}
S. Friedl and S. Vidussi,  
A survey of twisted Alexander polynomials, 
{\it The mathematics of knots}, 45--94, 
Contrib. Math. Comput. Sci., \textbf{1}, {\it Springer, Heidelberg,} 2011. 

\bibitem{IshiiKishimotoMoriuchiSuzuki12}
A. Ishii, K. Kishimoto, H. Moriuchi and M. Suzuki, 
A table of genus two handlebody-knots up to six crossings, 
\textit{J. Knot Theory Ramifications} \textbf{21} (2012), 1250035, 9 pp. 

\bibitem{Kamada89}
S. Kamada, 
Non-orientable surfaces in $4$-space, 
{\it Osaka J. Math.} \textbf{26} (1989), 367--385. 

\bibitem{KawauchiShibuyaSuzuki82}
A. Kawauchi, T. Shibuya and S. Suzuki, 
Descriptions on surfaces in four-space. I. Normal forms, 
\textit{Math. Sem. Notes Kobe Univ.} \textbf{10} (1982), 75--125.

\bibitem{Kinoshita72}
S. Kinoshita, 
On elementary ideals of polyhedra in the $3$-sphere, 
\textit{Pacific J. Math.} \textbf{42} (1972), 89--98. 

\bibitem{Lin01}
X. S. Lin, 
Representations of knot groups and twisted Alexander polynomials, 
\textit{Acta Math. Sin. (Engl. Ser.)} \textbf{17} (2001), 361--380.

\bibitem{Livingston95}
C. Livingston, \textit{Knotted symmetric graphs}, Proc. Amer. Math. Soc. \textbf{123} (1995), 963--967.

\bibitem{McAteeSilverWilliams01}
J. McAtee, D. S. Silver and S. G. Williams, 
Coloring spatial graphs, 
\textit{J. Knot Theory Ramifications} \textbf{10} (2001), 109--120. 

\bibitem{Sato10}
Y. Sato, 
On Alexander ideals for the fundamental group of a spatial graph complement (in Japanese), 
Master Thesis, Tokyo Woman's Christian University, 2010. 

\bibitem{scharlemann92}
M. Scharlemann, 
Some pictorial remarks on Suzuki's Brunnian graph, 
\textit{Topology '90 (Columbus, OH, 1990),} 351--354, 
Ohio State Univ. Math. Res. Inst. Publ., \textbf{1}, {\it de Gruyter, Berlin,} 1992. 

\bibitem{Suzuki84}
S. Suzuki, 
Almost unknotted $\theta_n$-curves in the $3$-sphere, 
\textit{Kobe J. Math.} \textbf{1} (1984), 19--22. 

\bibitem{Wada94}
M. Wada, 
Twisted Alexander polynomial for finitely presentable group, 
\textit{Topology} \textbf{33} (1994), 241--256. 

\bibitem{Yoshikawa94}
K. Yoshikawa, 
An enumeration of surfaces in four-space, 
\textit{Osaka J. Math.} \textbf{31} (1994), 497--522. 

\end{thebibliography}
\end{document}